\documentclass[10pt]{article}
\usepackage{latexsym,amsmath,amssymb,graphics,amscd}
\begin{document}

\newenvironment{proof}[1][Proof]{\textbf{#1.} }{\ \rule{0.5em}{0.5em}}

\newtheorem{theorem}{Theorem}[section]
\newtheorem{definition}[theorem]{Definition}
\newtheorem{lemma}[theorem]{Lemma}
\newtheorem{remark}[theorem]{Remark}
\newtheorem{proposition}[theorem]{Proposition}
\newtheorem{corollary}[theorem]{Corollary}
\newtheorem{example}[theorem]{Example}
\numberwithin{equation}{section}



\title{{\bf Invariant critical sets of conserved quantities}}
\author{Petre Birtea and Dan Com\u{a}nescu}

\date{}
\maketitle

\begin{abstract}
For a dynamical system we will construct various invariant sets
starting from its conserved quantities. We will give conditions under
which certain solutions of a nonlinear system are  also solutions
for a simpler dynamical system, for example when they are solutions for
a linear dynamical system. We will apply these results to the
example of Toda lattice.
\end{abstract}

{\bf MSC}: 34C45, 37Cxx, 70S10.

{\bf Keywords}: dynamical systems, invariant sets, conserved quantities, perturbed systems, Toda lattice.

\section{Introduction.}

A particle moving in the Newtonian gravitational field is known as
the Kepler problem. The equations of motion are
$$\ddot{x}=-\frac{x}{\parallel x\parallel^3},$$ where $x \in
\mathbb{R}^3\setminus\{0\}$ is the position vector.

It is well known that the motions for the Kepler problem are
planar. If we consider the plane $Ox_1x_2$, the equations of
motion becomes
$$
\left\{
\begin{array}{ll}
    \dot{x_1}=y_1\\
    \dot{x_2}=y_2\\
    \dot{y_1}=-\frac{x_1}{(x_1^2+x_2^2)^{3/2}}\\
    \dot{y_2}=-\frac{x_2}{(x_1^2+x_2^2)^{3/2}}\\
\end{array}
\right.
$$

These equations are Hamiltonian with the standard symplectic form
on $\mathbb{R}^4$ and the Hamiltonian function
$H=\frac{1}{2}(y_1^2+y_2^2)-\frac{1}{\sqrt{x_1^2+x_2^2}}$. From
Kepler's second law we have another conserved quantity given by
$A=x_1y_2-x_2y_1$.

For $a>0$, consider the following conserved quantity
$\emph{K}=H+\frac{1}{a^3}A$. After a straightforward computation
we will obtain that $\nabla\emph{K}=0$ if and only if
$y_1=\frac{1}{a^3}x_2$, and $y_2=-\frac{1}{a^3}x_1$ and
$||(x_1,x_2)||=a^2$. Equivalently, the set $\{\nabla\emph{K}=0\}$
is equal with the set
$\{(x_1,x_2,y_1,y_2)\mid(x_1,x_2)\cdot(y_1,y_2)=0,\,\,\,
\textit{and}\,\,\, ||(x_1,x_2)||=a^2,\,\,\,
\textit{and}\,\,\,||(y_1,y_2)||=\frac{1}{a}\,\,\,\textit{and}\,\,\,sgn(y_1)=sgn(x_2)
\}$ which is invariant under the dynamics and is filled with
solutions that represent uniform circular motions moving
clockwise. These particular motions are also solutions for the
linear Hamiltonian system
$$
\left\{
\begin{array}{ll}
    \dot{x_1}=-x_2\\
    \dot{x_2}=x_1\\
    \dot{y_1}=-y_2\\
    \dot{y_2}=y_1\,,\\
\end{array}
\right.
$$
where the Hamiltonian function is $A$.

Given the above analysis, we can raise at least two questions. Is
it true that for a dynamical system that admits conserved quantities, the set of points where the gradients of these conserved quantities
are zero, is an invariant set? When do two Hamiltonian systems have
common solutions? In what follows, we will give an answer to
these questions.

For the first question, the answer is positive and it is given in
section two where we will also discuss various generalizations of
this answer. In section three, we will present the conditions
under which we can give an answer to the second question. In
section four we will illustrate these results for the example of
Toda lattice. Detailed computations are presented in the Appendix.

\section{Invariant sets}

Let $f:\mathbb{R}^n\rightarrow \mathbb{R}^n$ be a $C^q$, $q\geq 1$
function which generates the differential equation
\begin{equation}\label{sys}
\dot{x}=f(x).
\end{equation}

We suppose that equation (\ref{sys}) admits a $C^q$ vectorial
conserved quantity $\mathbf{F}:\mathbb{R}^n\rightarrow \mathbb{R}^k$
with $k\leq n$. We denote by $F_1,...,F_k$ the components of
$\mathbf{F}$.
For $s\in \{0,...,k\}$ we introduce the sets:

\begin{equation}\label{Mi}
M_{(s)}^F=\{x\in \mathbb{R}^n\,|\,rank \nabla \mathbf{F}(x)=s\}
\end{equation}
where $\nabla\mathbf{F}(x)$ is the Jacobian matrix

\begin{equation}
\nabla \mathbf{F}(x)=
\left(%
\begin{array}{ccc}
  \frac{\partial F_1}{\partial x_1}(x) & ... & \frac{\partial F_1}{\partial x_n}(x) \\
  ... & ... & ... \\
  \frac{\partial F_k}{\partial x_1}(x) & ... & \frac{\partial F_k}{\partial x_n}(x) \\
\end{array}%
\right).
\end{equation}

For $r\in \{1,...,q\}$ we introduce the sets:
\begin{equation}\label{Nj}
N_{(r)}^F=\{x\in
\mathbb{R}^n\,|\,\partial^{\mathbf{\alpha}}F_i(x)=0,\,\,\forall
i\in\{1,...,k\},\,\,\,\forall \mathbf{\alpha}\in
\{1,...,n\}^l,\,\,\,\forall l\leq r \}
\end{equation}
where we note $\partial^{\mathbf{\alpha}}F_i=\frac{\partial ^l
F_i}{\partial x_{\alpha_1}...\partial x_{\alpha_l}}$ if
$\mathbf{\alpha}=(\alpha_1,...,\alpha_l)$.

\begin{remark}
We observe that
\begin{equation}
M_{(0)}^F=N_{(1)}^F=\{x\in \mathbb{R}^n\,|\,\frac{\partial
F_i}{\partial x_j}(x)=0,\,\,i\in \{1,...,k\},\,\,j\in
\{1,..,n\}\}.
\end{equation}
\end{remark}

The set $\{M_{(s)}^F\}$ with $s\in \{0,...,k\}$ is a partition of
$\mathbb{R}^n$. A critical point of $\mathbf{F}$ is a point in
$\mathbb{R}^n$ at which the rank of the matrix $\nabla
\mathbf{F}(x)$ is less than the maximum rank. A critical value is
the image under $\mathbf{F}$ of a critical point. The set of
critical points of $\mathbf{F}$ is
\begin{equation}
M_c^F=\cup_{s=0}^{k-1}M_{(s)}^F.
\end{equation}
Using Sard's Theorem (see \cite{sard}) we have that the set of
critical values $\mathbf{F}(M_c^F)$ is of $k$-dimensional measure
zero providing that $q\geq n-k+1$.

\begin{remark} We also have the obvious inclusions $N_{(q)}^F\subseteq N_{(q-1)}^F\subseteq ...\subseteq N_{(1)}^F.$

\end{remark}

\begin{theorem}\label{principal theorem}
The sets $M_{(s)}^F$ are invariant under the dynamics generated by
the differential equation (\ref{sys}).
\end{theorem}
\begin{proof} Because $\mathbf{F}$ is a conserved quantity, we have
$$\mathbf{F}(\Phi _t(x))=\mathbf{F}(x),$$
where $\Phi _t:\mathbb{R}^n\rightarrow \mathbb{R}^n$ is the flow
generated by (\ref{sys}).
Differentiating, we have
\begin{equation}\label{formula principala}
\nabla \mathbf{F}(\Phi_t(x))\nabla \Phi_t(x)=\nabla \mathbf{F}(x).
\end{equation}

As $\nabla \Phi_t(x)$ is an invertible matrix for any $x\in
\mathbb{R}^n$ which is not an equilibrium point for (\ref{sys}),
(see \cite{abra}), we have that
$$rank \nabla \mathbf{F}(\Phi_t(x))=rank \nabla \mathbf{F}(x),$$
which implies the stated result.
\end{proof}
\medskip

As a consequence we will obtain the following well known result which was applied for studying invariant sets of various mechanical systems, see for example \cite{irtegov}.

\begin{corollary}
The set of critical points of $\mathbf{F}$ is an invariant set of
the dynamics generated by the differential equation \eqref{sys}.
\end{corollary}

\begin{theorem}
The sets $N_{(r)}^F$ are invariant sets for the dynamics generated
by the differential equation (\ref{sys}).
\end{theorem}

\begin{proof}
Let $i\in \{1,...,k\}$, $l\in \{1,...,q\}$ and
$\mathbf{\alpha}=(\alpha_1,...,\alpha_l)\in \{1,...,n\}^l$. We
will prove by mathematical induction that
\begin{equation}\label{relatia baza}
\partial^{\mathbf{\alpha}}F_i(x)=\sum_{\beta_1,...,\beta_l=1}^n(\nabla\Phi_t)_{\alpha_1\beta_1}(x)...
(\nabla\Phi_t)_{\alpha_l\beta_l}(x)\partial^{\mathbf{\beta}}F_i(\Phi_t(x))+S_{i,\mathbf{\alpha}}(t,x),
\end{equation}
where $\mathbf{\beta}=(\beta_1,...,\beta_l)$ and
$S_{i,\mathbf{\alpha}}(t,x)$ is a sum with the property:
\emph{"all the terms contain a factor of the form
$\partial^{\mathbf{\gamma}}F_i(\Phi_t(x))$ with $|\gamma |<l$"}.

We will prove this result by mathematical induction with respect
to $l$. Componentwise the relation (\ref{formula principala})
implies our result for $l=1$.

Let $\mathbf{\alpha}'=(\alpha_1,...,\alpha_l,\alpha_{l+1})\in
\{1,...,n\}^{l+1}$ where
$\mathbf{\alpha}=(\alpha_1,...,\alpha_l)\in \{1,...,n\}^l$. Using
the induction hypothesis we have:
$$\partial^{\mathbf{\alpha}'}F_i(x)=\frac{\partial}{\partial x_{\alpha_{l+1}}}(\sum_{\beta_1,...,\beta_l=1}^n
(\nabla\Phi_t)_{\alpha_1\beta_1}(x)...(\nabla\Phi_t)_{\alpha_l\beta_l}(x)\partial^{\mathbf{\beta}}F_i(\Phi_t(x))
+\frac{\partial}{\partial
x_{\alpha_{l+1}}}(S_{i,\mathbf{\alpha}}(t,x)).$$ By a
straightforward computation we obtain
$$\frac{\partial}{\partial
x_{\alpha_{l+1}}}(\sum_{\beta_1,...,\beta_l=1}^n(\nabla\Phi_t)_{\alpha_1\beta_1}(x)...(\nabla\Phi_t)_{\alpha_l\beta_l}(x)
\partial^{\mathbf{\beta}}F_i(\Phi_t(x)))=$$
$$=\sum_{\beta_1,...,\beta_l=1}^n(\nabla\Phi_t)_{\alpha_1\beta_1}(x)...(\nabla\Phi_t)_{\alpha_l\beta_l}(x)
\frac{\partial}{\partial
x_{\alpha_{l+1}}}\partial^{\mathbf{\beta}}F_i(\Phi_t(x))+$$
$$+\sum_{\beta_1,...,\beta_l=1}^n \frac{\partial}{\partial
x_{\alpha_{l+1}}}((\nabla\Phi_t)_{\alpha_1\beta_1}(x)...(\nabla\Phi_t)_{\alpha_l\beta_l}(x))\partial^{\mathbf{\beta}}
F_i(\Phi_t(x))=$$
$$=\sum_{\beta_1,...,\beta_l,\beta_{l+1}=1}^n(\nabla\Phi_t)_{\alpha_1\beta_1}(x)...(\nabla\Phi_t)_{\alpha_l\beta_l}(x)
(\nabla\Phi_t)_{\alpha_{l+1}\beta_{l+1}}(x)
\partial^{\mathbf{\beta}'}F_i(\Phi_t(x))+$$
$$+\sum_{\beta_1,...,\beta_l=1}^n \frac{\partial}{\partial
x_{\alpha_{l+1}}}((\nabla\Phi_t)_{\alpha_1\beta_1}(x)...(\nabla\Phi_t)_{\alpha_l\beta_l}(x))\partial^{\mathbf{\beta}}
F_i(\Phi_t(x)).$$ We will note that
$$S_{i,\alpha '}(t,x)=\sum_{\beta_1,...,\beta_l=1}^n \frac{\partial}{\partial
x_{\alpha_{l+1}}}((\nabla\Phi_t)_{\alpha_1\beta_1}(x)...(\nabla\Phi_t)_{\alpha_l\beta_l}(x))
\partial^{\mathbf{\beta}}F_i(\Phi_t(x))+
\frac{\partial}{\partial
x_{\alpha_{l+1}}}(S_{i,\mathbf{\alpha}}(t,x)).$$ All the terms of
$S_{i,\alpha '}(t,x)$ contain a factor of the form
$\partial^{\mathbf{\gamma}}F_i(\Phi_t(x))$ with $|\gamma |<l+1$,
which had to be proved.

Let $\beta=(\beta_1,...,\beta_l)\in \{1,...,n\}^l$ and
$\nabla\Phi_t ^{-1}(x)$ be the inverse matrix of
$\nabla\Phi_t(x)$, where $x$ is not an equilibrium point for
(\ref{sys}). Consequently, we have
\begin{equation}\label{relatia baza 1}
\partial^{\mathbf{\beta}}F_i(\Phi_t(x))=\sum_{\alpha_1,...,\alpha_l=1}^n(\nabla\Phi_t)_{\beta_1\alpha_1}^{-1}(x)
...(\nabla\Phi_t)_{\beta_l\alpha_l}^{-1}(x)[\partial^{\mathbf{\alpha}}F_i(x)-S_{i,\mathbf{\alpha}}(t,x)].
\end{equation}

Also, we will prove by mathematical induction that the sets
$N_{(j)}^F$ are invariant under the dynamics generated by the
differential equation (\ref{sys}). For $j=1$ we have
$N_{(1)}^F=M_{(0)}^F$, which is an invariant set (see Theorem
\ref{principal theorem}). We suppose that for all $j\leq l$ the
sets $N_{(j)}^F$ are invariant. Let $x\in N_{(l+1)}^F$ be
arbitrary chosen. Using \eqref{relatia baza 1} for $l+1$-order of
derivation and the induction hypothesis, we deduce that
$\Phi_t(x)\in N_{(l+1)}^F,\,\,\,\forall t$. Summing up,
$N_{(j)}^F$ are invariant sets for all $j\in \{1,...,q\}$.
\end{proof}

\section{Finding solutions using simpler dynamics}

Let $F,G:\mathbb{R}^n\rightarrow \mathbb{R}$ be $C^q$ functions
and the differential equations:
\begin{equation}\label{unu}
    \dot{x}=f(x,\nabla F(x))
\end{equation}
and
\begin{equation}\label{doi}
    \dot{x}=f(x,\nabla G(x))
\end{equation}
where $f:\mathbb{R}^{2n}\rightarrow \mathbb{R}^n$ is a $C^q$
vectorial function.

For $x\in \mathbb{R}^n$, we denote with $\Phi_t^F(x)$ and
$\Phi_t^G(x)$ the solutions of (\ref{unu}) and (\ref{doi}) with
the initial conditions $\Phi_0^F(x)=x$ and $\Phi_0^G(x)=x$. We
will introduce the following set
$$E_1=\{x\in \mathbb{R}^n\,|\,\nabla F(x)=\nabla G(x)\}.$$

\begin{theorem}
If $F-G$ is a conserved quantity for (\ref{unu}) and $x\in E_1$,
then for all $t$ we have
$$\Phi_t^F(x)=\Phi_t^G(x).$$
\end{theorem}

\begin{proof} For $L=F-G$ we have the equality $E_1=M_{(0)}^L$. By Theorem (\ref{principal theorem}) the
 set $E_1$ is invariant under the dynamics of
(\ref{unu}). For $x\in E_1$ we have $\nabla F(\Phi_t^F(x))=\nabla
G(\Phi_t^F(x))$ for all $t$ and consequently
$$\frac{d}{dt}\Phi_t^F(x)=f(\Phi_t^F(x),\nabla F(\Phi_t^F(x))=f(\Phi_t^F(x),\nabla
G(\Phi_t^F(x)).$$ The above equality shows that $\Phi_t^F(x)$ is
also a solution for (\ref{doi}). Given the uniqueness of the
solutions, we obtain the desired equality.
\end{proof}
\bigskip

Next we will discus a particular case of the result presented
above. For this we will take $f(x,y)=h(x)+g(y)$ and $F\equiv 0$.
Thus we have the two dynamics
\begin{equation}\label{neperturbat}
\dot{x}=h(x)
\end{equation} and the perturbed dynamics
\begin{equation}\label{}
\dot{x}=h(x)+g(\nabla G(x)).
\end{equation}
We denote by $\Phi_t^h$ the flux for \eqref{neperturbat}.

\begin{corollary}\label{neperturb}
If $G$ is a conserved quantity for \eqref{neperturbat}, then for all
initial conditions $x\in\{x\in \mathbb{R}^n\,|\,\nabla G(x)=0\}$
we have
$$\Phi_t^h(x)=\Phi_t^G(x).$$
\end{corollary}

Another particular case is when the function $f$ in \eqref{unu}
and \eqref{doi} verifies the equality
$$<f(x,y),y>=0,$$
where $<\cdot,\cdot>$ is the Euclidian product on $\mathbb{R}^n$.
In this case $F$ is a conserved quantity for \eqref{unu} and $G$ is
a conserved quantity for \eqref{doi}. The above equality is verified
for the case when
$$f(x,y)=\Pi(x)y,$$
where $\Pi(x)$ is an antisymmetric matrix. This is true for all
almost Poisson manifolds (see \cite{ortega}). In this situation
the differential equations \eqref{unu} and \eqref{doi} become
\begin{equation}\label{poisson-unu}
    \dot{x}=\Pi(x)\nabla F(x)
\end{equation}
and
\begin{equation}\label{poisson-doi}
    \dot{x}=\Pi(x)\nabla G(x).
\end{equation}
We observe that $F$ is a conserved quantity for \eqref{poisson-doi}
if and only if $G$ is a conserved quantity for \eqref{poisson-unu}.

\begin{corollary}\label{simplectic} If $G$ is a conserved quantity for
\eqref{poisson-unu}, then for the initial conditions in $\{x\in
\mathbb{R}^n\,|\,\nabla F(x)=\nabla G(x)\}$ we have
$\Phi_t^F(x)=\Phi_t^G(x)$.
\end{corollary}

A particular case of Corollary \ref{simplectic} is when we have a
symplectic manifold with $G$ being a quadratic function. In this
case, the solutions of the Hamiltonian vector field $X_F$ starting
in $\{x\in \mathbb{R}^n\,|\,\nabla F(x)=\nabla G(x)\}$ are also
the solutions of the linear Hamiltonian system $X_G$.
\bigskip

Analogous results are valid for the more general case of the
vector valued conserved quantities and when the right side of
equations \eqref{unu} and \eqref{doi} depends on higher order
derivatives.

Firstly, we will introduce some notations. Let
$\mathbf{F},\mathbf{G}:\mathbb{R}^n\rightarrow \mathbb{R}^k$ be
$C^q$ vectorial functions with $q\geq 1$ and also $k\leq n$. If
$F_1,...,F_k$ are the components of $\mathbf{F}$ we denote with
\begin{equation*}
\Delta^1 \mathbf{F}(x)=(\frac{\partial F_1}{\partial x_1}(x),...,\frac{\partial F_1}{\partial x_n}(x),\frac{\partial F_2}{\partial
x_1}(x),...,\frac{\partial F_2}{\partial x_n}(x),...,\frac{\partial F_k}{\partial x_1}(x),...,\frac{\partial F_k}{\partial x_n}(x))\in
\mathbb{R}^{kn}.
\end{equation*}
and
\begin{equation*}
\Delta^r \mathbf{F}(x)=(...,\partial^{\alpha}F_i(x),...)\in
\mathbb{R}^{kn^r},
\end{equation*}
where $r\in \{1,...,q\}$, $i\in \{1,...,k\}$, $\alpha\in
\{1,...,n\}^r$ ($|\alpha|=r$) and the components appear in the
lexicographical order for $(i,\alpha)$ in $\mathbb{N}^{r+1}$.

For a fix $r\in \{1,...,q\}$, we will consider, as before, the two
differential equations
\begin{equation}\label{eqF}
\dot{x}=f_r(x,\Delta^1 \mathbf{F}(x),...,\Delta^r \mathbf{F}(x))
\end{equation}
and
\begin{equation}\label{eqG}
\dot{x}=f_r(x,\Delta^1 \mathbf{G}(x),...,\Delta^r \mathbf{G}(x)).
\end{equation}

Let $x\in \mathbb{R}^n$, we denote with $\Phi_t^\mathbf{F}(x)$ and
$\Phi_t^\mathbf{G}(x)$ the solutions of (\ref{eqF}) and
(\ref{eqG}) with initial conditions $\Phi_0^\mathbf{F}(x)=x$ and
$\Phi_0^\mathbf{G}(x)=x$. We will introduce the following set
$$E_r=\{x\in \mathbb{R}^n\,|\,\partial^{\alpha}\mathbf{F}(x)=\partial^{\alpha}\mathbf{G}(x),\,\,\,\forall |\alpha|\leq r\}.$$

\begin{theorem}\label{thprincipal}
If $\mathbf{F}-\mathbf{G}$ is a conserved quantity for \eqref{eqF}
and $x\in E_r$ then for all $t$ we have
$$\Phi_t^\mathbf{F}(x)=\Phi_t^\mathbf{G}(x).$$
\end{theorem}

Also the obvious extension of Corollary \ref{neperturb} takes
place.

\section{Invariant sets for Toda lattices}

\hspace{0.5cm} The Toda lattice describes the one-dimensional
motions of a chain of particles with nearest neighbor
interactions. For a chain of particles with the equal masses $m$,
Morikazu Toda came up with the choice of the interaction potential
$$V(r)=e^{-r}+r-1.$$
The system of motion reads explicitly
\begin{equation}
m\ddot{x_i}=e^{-(x_i-x_{i-1})}-e^{-(x_{i+1}-x_i)},\,\,\,i\in
\mathbb{Z}.
\end{equation}
This second order differential system is equivalent with the first
order differential system
\begin{equation}
\left\{%
\begin{array}{ll}
\dot{x_i}=u_i \\
m\dot{u_i}=e^{-(x_i-x_{i-1})}-e^{-(x_{i+1}-x_i)},\,\,\,i\in \mathbb{Z}, \\
\end{array}%
\right.
\end{equation}
where $u_i$ is the velocity of the particle $i$.

An equilibrium of the Toda lattice has the form
$$x_i=x_0+\lambda i,\,\,\,u_i=0,\,\,\,\lambda\in \mathbb{R}^*,\,\,\,x_0\in \mathbb{R},\,\,\,i\in \mathbb{Z}.$$
Let an equilibrium of the Toda lattice and $y_i=x_i-x_0-\lambda i$
the displacement of the $i$ particle from its equilibrium
position. The system in the variables $y_i$ and $u_i$ is
\begin{equation}\label{sistemul mecanic}
\left\{%
\begin{array}{ll}
\dot{y_i}=u_i \\
\dot{u_i}=\frac{e^{-\lambda}}{m}(e^{-(y_i-y_{i-1})}-e^{-(y_{i+1}-y_i)}),\,\,\,i\in \mathbb{Z} \\
\end{array}%
\right.
\end{equation}
Let us define
\begin{equation}\label{schimbare coordonate}
X_i:=\frac{e^{-\lambda}}{m}e^{-(y_{i+1}-y_i)},
\end{equation}
then the equations of motion become
\begin{equation}
\left\{%
\begin{array}{ll}\label{sistemul de baza}
\dot{X_i}=X_i(u_i-u_{i+1}) \\
\dot{u_i}=X_{i-1}-X_i,\,\,\,i\in \mathbb{Z}\,. \\
\end{array}%
\right.
\end{equation}

The following particular cases are interesting:

\noindent {\bf 1.} the case of a periodic lattice,
$X_{i+n}=X_i\,\,\forall i\in \mathbb{Z}$,

\noindent {\bf 2.} the case of a non-periodic lattice with the
boundary conditions $X_0=0$ (correspond to formally setting
$y_0=-\infty$) and $X_n=0$ (correspond to formally setting
$y_{n+1}=\infty$).

In both cases we investigate the motions of the particles $1$ to
$n$ ($n\in \mathbb{N}^*$).

\subsection{The case of a periodic lattice}

\hspace{0.5cm} In this case {\it M. H$\acute{e}$non} proves in \cite{henon} that
the following expressions are scalar conserved quantities
\begin{equation}
I_m=\sum u_{i_1}...u_{i_k}(-X_{j_1})...(-X_{j_l}),
\end{equation}
where $m\in \{1,...,n\}$ and the summation are extended to all
terms which satisfy the following conditions:

\noindent (i) the indices $i_1,...,i_k,j_1,j_1+1,...,j_l,j_l+1$,
which appear in the term (either explicitly, or implicitly through
a factor $X_j$) are all different (modulo $n$);

\noindent (ii) the number of these indices is $m$, i.e. $k+2l=m$.
Two terms differing only in the order of factors are not
considered different, and therefore only one of them appears in
the sum.

In \cite{flaschkaperiodic}, Flaschka has proved that the above
functions are conserved quantities using a Lax formulation. This was
generalized to arbitrary Lie algebras by Adler \cite{adler} and
Kostant \cite{kostant}.

The first three scalar conserved quantities, depending on the variables $(X_1,...,X_n,u_1,...,u_n)$, are
\begin{equation}
I_1=\sum_{1\leq i\leq n} u_i
\end{equation}
\begin{equation}
I_2=\sum _{1\leq i_1<i_2\leq n} u_{i_1}u_{i_2}-\sum_{1\leq j\leq
n} X_j
\end{equation}
\begin{equation}
I_3=\sum _{1\leq i_1<i_2<i_3\leq n}
u_{i_1}u_{i_2}u_{i_3}-\sum_{1\leq i,j\leq n,\ j\neq i,\,j\neq i-1
(mod\,n)} u_iX_j,\,\,\,(X_0=X_n).
\end{equation}

We will introduce the vectorial conserved quantities,
\begin{equation}\label{}
\mathbb{I}_{12}=(I_1,I_2),\,\,\mathbb{I}_{13}=(I_1,I_3),\,\,\mathbb{I}_{23}=(I_2,I_3),\,\,\mathbb{I}_{123}=(I_1,I_2,I_3).
\end{equation}

{\bf I. The case $n$ odd}

For this case we obtain as invariant sets only subsets of the set of equilibrium points or the empty set. More precisely, $M_{(0)}^{I_{1}}=M_{(0)}^{I_{2}}=M_{(0)}^{\mathbb{I}_{12}}=M_{(1)}^{\mathbb{I}_{12}}=
M_{(0)}^{\mathbb{I}_{13}}=M_{(0)}^{\mathbb{I}_{23}}=M_{(0)}^{\mathbb{I}_{123}}=M_{(1)}^{\mathbb{I}_{123}}=\emptyset$. The following are subsets of the set of equilibrium points:
\begin{equation*}
M_{(0)}^{I_3}=\{(0,...,0,0,...,0)\},
\end{equation*}
\begin{equation*}
M_{(1)}^{\mathbb{I}_{13}}=\{(X,...,X,0,...,0)\,|\,X\in
\mathbb{R}\},
\end{equation*}
\begin{equation*}
M_{(1)}^{\mathbb{I}_{23}}=\{(-\frac{n-1}{2}u^2,...,-\frac{n-1}{2}u^2,u,...,u)\,|\,u\in
\mathbb{R}\},
\end{equation*}
\begin{equation*}\label{}
M_{(2)}^{\mathbb{I}_{123}}=\{(X,...,X,u,...,u)\,|\,X,u\in
\mathbb{R}\}.
\end{equation*}

{\bf II. The case $n$ even}

We have, $M_{(0)}^{I_{1}}=M_{(0)}^{I_{2}}=M_{(0)}^{\mathbb{I}_{12}}=M_{(1)}^{\mathbb{I}_{12}}=
M_{(0)}^{\mathbb{I}_{13}}=M_{(0)}^{\mathbb{I}_{23}}=M_{(0)}^{\mathbb{I}_{123}}=M_{(1)}^{\mathbb{I}_{123}}=\emptyset$.
As nontrivial invariant sets we have the following:
\begin{equation*}
M_{(0)}^{I_3}=\{(X_1,X_2,...,X_1,X_2,u_1,u_2,...,u_1,u_2)\,|\,X_1+X_2=u_1u_2,\,u_1+u_2=0\},
\end{equation*}
\begin{equation*}
M_{(1)}^{\mathbb{I}_{13}}=\{(X_1,X_2,...,X_1,X_2,u_1,u_2,...,u_1,u_2)\,|\,u_1+u_2=0\},
\end{equation*}
\begin{equation*}
  M_{(1)}^{\mathbb{I}_{23}}=\{(X_1,X_2,...,X_1,X_2,u_1,u_2,...,u_1,u_2)\,|\,X_1+X_2=-\frac{n}{4}(u_1+u_2)^2+u_1u_2\},
\end{equation*}
\begin{equation*}\label{}
M_{(2)}^{\mathbb{I}_{123}}=\{(X_1,X_2,...,X_1,X_2,u_1,u_2,...,u_1,u_2)\,|\,X_1,X_2,u_1,u_2\in
\mathbb{R}\}.
\end{equation*}
We obtain $M_{(2)}^{\mathbb{I}_{123}}$ as the largest
invariant set and the restricted dynamics is the dynamics of two
particles
\begin{equation*}\label{}
    \left\{%
\begin{array}{ll}
    \dot{X}_1=X_1(u_1-u_2) \\
    \dot{X}_2=X_2(u_2-u_1) \\
    \dot{u}_1=X_2-X_1 \\
    \dot{u}_2=X_1-X_2 \\
\end{array}%
\right.
\end{equation*}
On the invariant set $M_{(1)}^{\mathbb{I}_{23}}$ we have the above
dynamics subject to $X_1+X_2=-\frac{n}{4}(u_1+u_2)^2+u_1u_2$. On
the invariant set $M_{(1)}^{\mathbb{I}_{13}}$ we have the above
dynamics subject to $u_1+u_2=0$ and on the invariant set
$M_{(0)}^{I_{3}}$ we have the above dynamics subject to
$u_1+u_2=0$ and $X_1+X_2=u_1u_2$.

For the sets $M_{(0)}^{I_3}$ and $M_{(1)}^{\mathbb{I}_{23}}$ the
variables $X_i$ have to take also negative values which from a
mathematical point of view is correct and can be the solutions for
the system \eqref{sistemul de baza}. As the mechanical system is
given by \eqref{sistemul mecanic} and we do the change of
variables \eqref{schimbare coordonate}, the variables $X_i$ have
to be strictly positive in order to have a physical meaning.
Consequently, from a mechanical point of view only the sets
$M_{(1)}^{\mathbb{I}_{13}}$ and $M_{(2)}^{\mathbb{I}_{123}}$ are
meaningful. The computations can be found in the Appendix.

\subsection{The case of non-periodic lattice}

It is known that if we have the matrices
$$L=\left(%
\begin{array}{ccccc}
  u_1 & X_1 & 0 & ... & 0 \\
  1 & u_2 & X_2 & ... & 0 \\
  ... & ... & ... & ... & ... \\
  0 & ... & 1 & u_{n-1} & X_{n-1} \\
  0 & ... & 0 & 1 & u_n \\
\end{array}%
\right)\,\,\,\texttt{and}\,\,\,B=\left(%
\begin{array}{ccccc}
  0 & -X_1 & 0 & ... & 0 \\
  0 & 0 & -X_2 & ... & 0 \\
  ... & ... & ... & ... & ... \\
  0 & ... & 0 & 0 & -X_{n-1} \\
  0 & ... & 0 & 0 & 0 \\
\end{array}%
\right),$$ then the system \eqref{sistemul de baza}, in this case,
has the Lax form
\begin{equation}
\dot{L}=[B,L],
\end{equation}
where $[B,L]=BL-LB$. Using the Flaschka theorem (see
\cite{flaschkaperiodic}) we have the following scalar conserved quantities depending on the variables $(X_1,...,X_{n-1},u_1,...,u_n)$
\begin{equation}
F_k=\frac{1}{k}tr(L^k),\,\,\,k\in\{1,...,n\}.
\end{equation}
For $k\in\{1,2,3\}$ we have
$$F_1=\sum_{i=1}^n u_i$$
$$F_2=\sum_{i=1}^n (X_i+\frac{u_i^2}{2})$$
$$F_3=\sum_{i=1}^{n-1} X_i(u_i+u_{i+1})+\frac{1}{3}\sum_{i=1}^n u_i^3.$$
We will introduce the vectorial conserved quantities,
\begin{equation}\label{}
\mathbb{F}_{12}=(F_1,F_2),\,\,\mathbb{F}_{13}=(F_1,F_3),\,\,\mathbb{F}_{23}=(F_2,F_3),\,\,\mathbb{F}_{123}=(F_1,F_2,F_3).
\end{equation}
As before we will distinguish two cases.\vspace{2mm}

{\bf I. The case $n$ odd}

 In this case we obtain as invariant sets only subsets of the set of equilibrium points or the empty set. More precisely, $M_{(0)}^{F_{1}}=M_{(0)}^{F_{2}}=M_{(0)}^{\mathbb{F}_{12}}=M_{(1)}^{\mathbb{F}_{12}}=
M_{(0)}^{\mathbb{F}_{13}}=M_{(0)}^{\mathbb{F}_{23}}=M_{(0)}^{\mathbb{F}_{123}}=M_{(1)}^{\mathbb{F}_{123}}=\emptyset$. The following are subsets of the set of equilibrium points:
\begin{equation*}
M_{(0)}^{F_3}=M_{(1)}^{\mathbb{F}_{13}}=\{(\underbrace{0,...,0}_{n-1},\underbrace{0,...,0}_{n})\},
\end{equation*}
\begin{equation*}
M_{(1)}^{\mathbb{F}_{23}}=\{(\underbrace{0,...,0}_{n-1},\underbrace{u_1,0,...,u_1,0,u_1}_n)\,|\,u_1\in \mathbb{R}\}\cup\{(\underbrace{0,...,0}_{n-1},
\underbrace{0,u_2,...,0,u_2,0}_n)\,|\,u_2\in\mathbb{R}\},
\end{equation*}
\begin{equation*}\label{}
M_{(2)}^{\mathbb{F}_{123}}=\{(\underbrace{0,...,0}_{n-1},\underbrace{u_1,u_2,...,u_1}_{n})\,|\,u_1,u_2\in \mathbb{R}\}.
\end{equation*}

{\bf II. The case $n$ even}

We have, $M_{(0)}^{F_{1}}=M_{(0)}^{F_{2}}=M_{(0)}^{\mathbb{F}_{12}}=M_{(1)}^{\mathbb{F}_{12}}=
M_{(0)}^{\mathbb{F}_{13}}=M_{(0)}^{\mathbb{F}_{23}}=M_{(0)}^{\mathbb{F}_{123}}=M_{(1)}^{\mathbb{F}_{123}}=\emptyset$.
As nontrivial invariant sets we have the following:
\begin{equation*}
M_{(0)}^{F_3}=\{(\underbrace{X,0,...,X,0,X}_{n-1},\underbrace{u_1,u_2,...,u_1,u_2}_n)\,|\,u_1+u_2=0,\,X=u_1u_2\},
\end{equation*}
\begin{equation*}
M_{(1)}^{\mathbb{F}_{13}}=\{(\underbrace{X,0,...,X,0,X}_{n-1},\underbrace{u_1,u_2,...,u_1,u_2}_n)\,|\,u_1+u_2=0\},
\end{equation*}
\begin{equation*}
  M_{(1)}^{\mathbb{F}_{23}}=\{(\underbrace{X,0,...,X,0,X}_{n-1},\underbrace{u_1,u_2,...,u_1,u_2}_n)\,|\,X=u_1u_2\},
\end{equation*}
\begin{equation*}\label{}
M_{(2)}^{\mathbb{F}_{123}}=\{(\underbrace{X,0,...,X,0,X}_{n-1},\underbrace{u_1,u_2,...,u_1,u_2}_n)\,|\,X,u_1,u_2\in \mathbb{R}\}.
\end{equation*}
We obtain $M_{(2)}^{\mathbb{F}_{123}}$ as the largest
invariant set and the restricted dynamics is given by
\begin{equation*}\label{}
    \left\{%
\begin{array}{ll}
    \dot{X}=X(u_1-u_2) \\
    \dot{u}_1=-X \\
    \dot{u}_2=X \\
\end{array}%
\right.
\end{equation*}
On the invariant set $M_{(1)}^{\mathbb{F}_{23}}$ we have the above
dynamics subject to $X=u_1u_2$. On the invariant set
$M_{(1)}^{\mathbb{F}_{13}}$ we have the above dynamics subject to
$u_1+u_2=0$ and on the invariant set $M_{(0)}^{F_{3}}$ we have the
above dynamics subject to $u_1+u_2=0$ and $X=u_1u_2$.

For the set $M_{(2)}^{\mathbb{F}_{123}}$ the variables $X_i$ with
$i$ even are all equal with zero which from a mathematical point
of view is correct. As before, the mechanical system is given by
\eqref{sistemul mecanic} and we do the change of variables
\eqref{schimbare coordonate}. Consequently, the variables $X_i$
have to be strictly positive in order to have a physical meaning. The computations can be found in the Appendix.

\section{Appendix}

{\bf The computations for the case of periodic lattice.}\vspace{2mm}

\noindent We will make the notations
\begin{equation}
U=\sum_{1\leq i\leq n}u_i,\,\,\,V=\sum_{1\leq i_1<i_2\leq
n}u_{i_1}u_{i_2},\,\,\,Y=\sum_{1\leq j\leq n}X_j.
\end{equation}
We observe that
\begin{equation}\label{2V}
2V=U^2-\sum_{i=1}^n u_i^2.
\end{equation}
With these notations we have the following,
\begin{equation}
\nabla I_1=(0,...,0,1,...,1)
\end{equation}
\begin{equation}
\nabla
I_2=(-1,...,-1,\underbrace{U-u_1}_{n+1},...,\underbrace{U-u_k}_{n+k},...,\underbrace{U-u_n}_{2n})
\end{equation}
\begin{equation}
\nabla
I_3=(...,\underbrace{-(U-u_k-u_{k+1})}_{k},...,\underbrace{V-u_k(U-u_k)-(Y-X_{k-1}-X_k)}_{n+k},...).
\end{equation}

\noindent {\bf The study of $M_{(0)}^{I_3}$}.

\noindent The elements of $M_{(0)}^{I_3}$ are the solutions of the algebraic
system
\begin{equation}\label{sistemul1}
\left\{%
\begin{array}{ll}
U-u_k-u_{k+1}=0, \\
V-u_k(U-u_k)-(Y-X_{k-1}-X_k)=0,\,\,\,\forall k\in\{1,...,n\}.
\end{array}%
\right.
\end{equation}
Adding the first $n$ equations, we obtain $U=0$ and
$u_1+u_2=u_2+u_3=...=u_{n-1}+u_n=u_n+u_1.$ We deduce the following
results:\vspace{2mm}

\noindent {\bf The case $n\in 2\mathbb{N}+1$}. In this
situation we have $u_1=u_2=...=u_n=0$. Adding the last $n$
equations of \eqref{sistemul1} we obtain
$Y=0\,\,\texttt{and}\,\,\,X_1+X_2=X_2+X_3=...=X_{n-1}+X_n=X_n+X_1.$
It implies that $X_1=...=X_n=0$ and consequently
\begin{equation*}
M_{(0)}^{I_3}=\{(0,...,0,0,...,0)\}.
\end{equation*}

\noindent {\bf The case $n\in 2\mathbb{N}$}. For this situation
$u_i=(-1)^{i+1}u,\,\,\,u\in \mathbb{R}$.
In this case, using \eqref{2V}, we have
$V=-\frac{n}{2}u^2$. The last $n$ equations of \eqref{sistemul1}
become
$$-Y+X_{k-1}+X_k=(\frac{n}{2}-1)u^2\,\,\,\forall k\in \{1,...,n\}.$$
Making the addition we have $Y=-\frac{n}{2}u^2$ and
consequently $X_1+X_2=X_2+X_3=...=X_{n-1}+X_n=X_n+X_1=-u^2$ which
implies
$X_1=X_3=...=X_{n-1},\,\,\,\texttt{and}\,\,\,X_2=X_4=...=X_{2n}.$
In this case we have
\begin{equation*}
M_{(0)}^{I_3}=\{(X_1,X_2,...,X_1,X_2,u_1,u_2,...,u_1,u_2)\,|\,X_1+X_2=u_1u_2,\,u_1+u_2=0\}.
\end{equation*}

\noindent {\bf The study of $M_{(0)}^{\mathbb{I}_{ij}}$,
$M_{(1)}^{\mathbb{I}_{ij}}$ with $(i,j)\in
\{(1,2),(1,3),(2,3)\}.$\vspace{2mm} }

\noindent A point $(X_1,...,X_n,u_1,...,u_n)\in M_{(1)}^{\mathbb{I}_{13}}$
if and only if we have, for all $k,q\in \{1,...,n\}$,
\begin{equation}
\left\{%
\begin{array}{ll}
U-u_k-u_{k+1}=0 \\
V-u_k(U-u_k)-(Y-X_{k-1}-X_k)=V-u_q(U-u_q)-(Y-X_{q-1}-X_q). \\\end{array}%
\right.
\end{equation}
Adding the first $n$ equations we obtain $U=0$,
$u_1+u_2=u_2+u_3=...=u_{n-1}+u_n=u_n+u_1$ and consequently
$u_i=(-1)^{i+1}u,\,\,\,u\in \mathbb{R}$. The last $n$ equations
become $X_1+X_2=X_2+X_3=...=X_{n-1}+X_n=X_n+X_1$. \vspace{2mm}

\noindent {\bf The case $n\in 2\mathbb{N}+1$}. For this case
note that $u_1=u_2=...=u_n=0$ and $X_1=...=X_n=X\in \mathbb{R}$.
In this case we have
\begin{equation*}
M_{(1)}^{\mathbb{I}_{13}}=\{(X,...,X,0,...,0)\,|\,X\in
\mathbb{R}\}.
\end{equation*}

\noindent {\bf The case $n\in 2\mathbb{N}$}. It is easy to
see that $X_1=X_3=...=X_{n-1},\,\,\,X_2=X_4=...=X_{n}.$ In this
case we have
\begin{equation*}
M_{(1)}^{\mathbb{I}_{13}}=\{(X_1,X_2,...,X_1,X_2,u_1,u_2,...,u_1,u_2)\,|\,u_1+u_2=0\}.
\end{equation*}
A point belongs to the set $M_{(1)}^{\mathbb{I}_{23}}$ if and only
if
\begin{equation}\label{conditii}
det(A_{kq})=0,\,\,\,\det(B_{kq})=0,\,\,\,det(C_{kq})=0,\,\,\,\forall
k,q\in \{1,...,n\},
\end{equation}
where
\begin{equation*}
A_{kq}=\left(%
\begin{array}{cc}
  -1 & -1 \\
  -(U-u_k-u_{k+1}) & -(U-u_q-u_{q+1}) \\
\end{array}%
\right)
\end{equation*}

\begin{equation*}
B_{kq}=\left(%
\begin{array}{cc}
  -1 & U-u_q \\
  -(U-u_k-u_{k+1}) & V-u_q(U-u_q)-(Y-X_{q-1}-X_q) \\
\end{array}%
\right)
\end{equation*}

\begin{equation*}
C_{kq}=\left(%
\begin{array}{cc}
 U-u_k  & U-u_q \\
 V-u_k(U-u_k)-(Y-X_{k-1}-X_k) & V-u_q(U-u_q)-(Y-X_{q-1}-X_q) \\
\end{array}%
\right).
\end{equation*}
Using the expression of $A_{kq}$ we deduce that
$u_1+u_2=u_2+u_3=...=u_{n-1}+u_n=u_n+u_1$. \vspace{2mm}

\noindent {\bf The case $n\in 2\mathbb{N}+1$.} In this case we
have $u_1=u_2=...=u_n=u$, and $U=nu$, and $V=\frac{(n-1)n}{2}u^2$
and
$$B_{kq}=\left(%
\begin{array}{cc}
  -1 & (n-1)u \\
  -(n-2)u & \frac{(n-2)(n-1)}{2}u^2-(Y-X_q-X_{q-1}) \\
\end{array}%
\right).$$ Using the condition that $det (B_{kq})=0$, we obtain
$$Y-X_q-X_{q-1}=-\frac{(n-2)(n-1)}{2}u^2,\,\,\,\forall q\in \{1,...,n\}.$$
Adding these relations we have
$$Y=-\frac{(n-1)n}{2}u^2,\,\,\,\texttt{and}\,\,\,X_1=X_2=....=X_n=X\,\,\,\texttt{and}\,\,\,X=-\frac{n-1}{2}u^2.$$
With this relation we have
$$C_{kq}=\left(%
\begin{array}{cc}
 (n-1)u  & (n-1)u \\
 (n-2)(n-1)u^2 & (n-2)(n-1)u^2 \\
\end{array}%
\right).$$ We observe that the equality $det (C_{kq})=0$ is
verified.
In conclusion we have
\begin{equation*}
M_{(1)}^{\mathbb{I}_{23}}=\{(-\frac{n-1}{2}u^2,...,-\frac{n-1}{2}u^2,u,...,u)\,|\,u\in
\mathbb{R}\}.
\end{equation*}

\noindent {\bf The case $n\in 2\mathbb{N}$.} In this case
$u_1=u_3=...=u_{n-1},\,\,\,u_2=u_4=....=u_n$ and we have
$U=\frac{n}{2}(u_1+u_2)$. Using \eqref{2V} we obtain
$V=\frac{n(n-2)}{8}(u_1^2+u_2^2)+\frac{n^2}{4}u_1u_2$. From the
relation $det(B_{kq})=0$, we have
\begin{equation}\label{relatieY}
Y-X_{q-1}-X_q=V-(U-u_q)(\frac{n-2}{n}U+u_q).
\end{equation}
Consequently,
$X_1+X_2=X_3+X_4=...=X_{n-1}+X_n\,\,\,\texttt{and}\,\,\,X_2+X_3=X_4+X_5=...=X_{n}+X_1$
which further implies that
$X_1=X_3=...=X_{n-1},\,\,\,X_2=X_4=...=X_n\,\,\,\texttt{and}\,\,\,Y=\frac{n}{2}(X_1+X_2).$
By substitution into \eqref{relatieY} we obtain
$X_1+X_2=-\frac{n}{4}(u_1+u_2)^2+u_1u_2.$
We have
\begin{equation*}\begin{array}{c}
  M_{(1)}^{\mathbb{I}_{23}}=\{(X_1,X_2,...,X_1,X_2,u_1,u_2,...,u_1,u_2)\,|\,X_1+X_2=-\frac{n}{4}(u_1+u_2)^2+u_1u_2\}.
  \end{array}
\end{equation*}

\noindent {\bf The study of $M_{(2)}^{\mathbb{I}_{123}}$}.
\vspace{2mm}

\noindent We introduce, for
$k,q,r\in \{1,...,n\}$, the matrices
\begin{equation*}\label{}
A_{kqr}=\left(%
\begin{array}{ccc}
  0 & 0 & 1 \\
  -1 & -1 & U-u_r \\
  -(U-u_k-u_{k+1}) & -(U-u_q-u_{q+1}) & V_r \\
\end{array}%
\right)
\end{equation*}
\begin{equation*}\label{}
B_{kqr}=\left(%
\begin{array}{ccc}
  0 & 1 & 1 \\
  -1 & U-u_q & U-u_r \\
  -(U-u_k-u_{k+1}) & V_q & V_r \\
\end{array}%
\right)
\end{equation*}
\begin{equation*}\label{}
C_{kqr}=\left(%
\begin{array}{ccc}
  1 & 1 & 1 \\
  U-u_k & U-u_q & U-u_r \\
  V_k & V_q & V_r \\
\end{array}%
\right),
\end{equation*}
where $V_s=V-u_s(U-u_s)-(Y-X_{s-1}-X_s)$.
We observe that $rank (\nabla
\mathbb{I}_{123})=2\,\,\Leftrightarrow\,\,det (A_{kqr})=det
(B_{kqr})=det (C_{kqr})=0\,\,\forall k,q,r\in\{1,...,n\}$. The
equations $det (A_{kqr})=0\,\,\forall k,q,r\in\{1,...,n\}$ give us
$$u_1+u_2=u_2+u_3=...=u_{n-1}+u_n=u_n+u_1.$$

\noindent {\bf The case $n\in 2\mathbb{N}+1$}. We have $u:=u_1=...=u_n$.
With this notation we obtain $U=nu$, and $V=\frac{n(n-1)}{2}u^2$
and $V_k=-\frac{n(n-1)}{2}u^2-(Y-X_{k-1}-X_k)$. It is easy to see
that $det (C_{kqr})=0$. We have the equivalences
$$det (B_{kqr})=0\,\,\Leftrightarrow\,\,det \left(%
\begin{array}{ccc}
  0 & 0 & 1 \\
  -1 & 0 & (n-1)u \\
  V_k & V_q-V_r & V_r \\
\end{array}%
\right)=0$$
$$\Leftrightarrow\,\,V_r=V_q\,\,\Leftrightarrow\,\,X_{r-1}+X_r=X_{q-1}+X_q.$$
Because $n\in 2\mathbb{N}+1$, we obtain $X_1=X_2=...=X_n.$ In this
case we have
\begin{equation*}\label{}
M_{(2)}^{\mathbb{I}_{123}}=\{(X,...,X,u,...,u)\,|\,X,u\in
\mathbb{R}\}.
\end{equation*}

\noindent {\bf The case $n\in 2\mathbb{N}$}. We have
$u_1=u_3=...=u_{n-1},\,\,\,u_2=u_4=....=u_n$. By calculus we
obtain
$$det(C_{kqr})=0\,\,\Leftrightarrow\,\,det \left(%
\begin{array}{ccc}
  1 & 0 & 0 \\
  U-u_k & u_k-u_q & u_k-u_r \\
  V_k & V_q-V_k & V_r-V_k \\
\end{array}%
\right)=0.$$ If $q=k+1$ and $r=k+2$ we obtain $V_k=V_{k+2}$, which
implies that
$X_1+X_2=X_3+X_4=...=X_{n-1}+X_n$ and $X_2+X_3=X_4+X_5=...=X_n+X_1$.
Consequently, we have
$X_1=X_3=...=X_{n-1}$ and $X_2=X_4=...=X_n.$
We observe that if $s-t\in 2\mathbb{Z}$, then $u_s=u_t$ and
$V_s=V_t$, which implies that $det(C_{kqr})=0\,\,\,\forall
k,q,r\in \{1,...,n\}$.
For the matrices $B_{kqr}$ we have the following properties:

\noindent i) $det B_{111}=detB_{112}=0$ (by calculus).

\noindent ii) $det B_{kqr}=-det B_{krq}$.

\noindent iii) If $q_1-q_2\in 2\mathbb{N}$ and $r_1-r_2\in
2\mathbb{N}$ then $det B_{k_1q_1r_1}=det B_{k_2q_2r_2}$.

\noindent Using these properties we deduce that $det B_{kqr}=0,\,\,\,\forall
k,q,r\in \{1,...,n\}$ and we obtain
\begin{equation*}\label{}
M_{(2)}^{\mathbb{I}_{123}}=\{(X_1,X_2,...,X_1,X_2,u_1,u_2,...,u_1,u_2)\,|\,X_1,X_2,u_1,u_2\in
\mathbb{R}\}.
\end{equation*}

\noindent {\bf The computations for the case of non-periodic lattice.}\vspace{2mm}

\noindent If we consider the variables $(X_1,...,X_{n-1},u_1,...,u_n)$ we
obtain
\begin{equation}
\nabla F_1=(\underbrace{0,...,0}_{n-1},\underbrace{1,...,1}_n)
\end{equation}
\begin{equation}
\nabla F_2=(\underbrace{1,...,1}_{n-1},u_1,...,u_n)
\end{equation}
\begin{equation}
\nabla
F_3=(u_1+u_2,...,u_{n-1}+u_n,X_1+u_1^2,X_2+X_1+u_2^2,...,X_{n-2}+X_{n-1}+u_{n-1}^2,X_{n-1}+u_n^2).
\end{equation}

\noindent {\bf The study of $M_{(0)}^{F_3}$}.

\noindent The elements of $M_{(0)}^{F_3}$ are the solutions of the system
$$\left\{%
\begin{array}{ll}
u_1+u_2=u_2+u_3=...=u_{n-1}+u_n=0 \\
X_1+u_1^2=X_2+X_1+u_2^2=...=X_{n-2}+X_{n-1}+u_{n-1}^2=X_{n-1}+u_n^2=0. \\\end{array}%
\right.$$

\noindent {\bf The case $n\in 2\mathbb{N}+1$}. We obtain
$$M_{(0)}^{F_3}=\{(\underbrace{0,...,0}_{n-1},\underbrace{0,...,0}_n)\}.$$

\noindent {\bf The case $n\in 2\mathbb{N}$}. In this
situation $u_1=u_3=...=u_{n-1}=u,\,\,\,u_2=u_4=...=u_n=-u$ and
$X_k=-u^2$ if $k\in 2\mathbb{N}+1$ and $X_k=0$ if $k\in
2\mathbb{N}$. We have
$$M_{(0)}^{F_3}=\{(\underbrace{X,0,...,X,0,X}_{n-1},\underbrace{u_1,u_2,...,u_1,u_2}_n)\,|\,u_1+u_2=0,\,X=u_1u_2\}.$$

\noindent {\bf The study of $M_{(0)}^{\mathbb{F}_{ij}}$,
$M_{(1)}^{\mathbb{F}_{ij}}$ with $(i,j)\in \{(1,2),(1,3),(2,3)\}$
}.\vspace{2mm}

\noindent The elements of $M_{(1)}^{\mathbb{F}_{13}}$ are the solutions of
the system
$$\left\{%
\begin{array}{ll}
u_1+u_2=u_2+u_3=...=u_{n-1}+u_n=0 \\
X_1+u_1^2=X_2+X_1+u_2^2=...=X_{n-2}+X_{n-1}+u_{n-1}^2=X_{n-1}+u_n^2. \\\end{array}%
\right.$$

\noindent {\bf The case $n\in 2\mathbb{N}+1$}. We obtain that
$$M_{(1)}^{\mathbb{F}_{13}}=\{(\underbrace{0,...,0}_{n-1},\underbrace{0,...,0}_n)\}.$$

\noindent {\bf The case $n\in 2\mathbb{N}$}. In this
situation $u_1=u_3=...=u_{n-1}=u,\,\,\,u_2=u_4=...=u_n=-u$ and
$X_k=X$ if $k\in 2\mathbb{N}+1$ and $X_k=0$ if $k\in 2\mathbb{N}$.
We have
$$M_{(1)}^{\mathbb{F}_{13}}=\{(\underbrace{X,0,...,X,0,X}_{n-1},\underbrace{u_1,u_2,...,u_1,u_2}_n)\,|\,u_1+u_2=0\}.$$

\noindent For a point of the set $M_{(1)}^{\mathbb{F}_{23}}$ we have
\begin{equation}\label{conditii}
det(A_{kq})=0,\,\,\,\det(B_{kq})=0,\,\,\,det(C_{kq})=0,\,\,\,\forall
k,q
\end{equation}
where
\begin{equation}
A_{kq}=\left(%
\begin{array}{cc}
1 & 1 \\
u_k+u_{k+1} & u_q+u_{q+1} \\
\end{array}%
\right),\,\,\,k,q\in\{1,...,n-1\}
\end{equation}

\begin{equation}
B_{kq}=\left(%
\begin{array}{cc}
  1 & u_q \\
  u_k+u_{k+1} & X_{q-1}+X_q+u_q^2 \\
\end{array}%
\right),\,\,\,k\in\{1,...,n-1\}\,\,\,\texttt{and}\,\,\,q\in\{1,...,n\}
\end{equation}

\begin{equation}
C_{kq}=\left(%
\begin{array}{cc}
u_k & u_q \\
X_{k-1}+X_k+u_k^2 & X_{q-1}+X_q+u_q^2 \\
\end{array}%
\right),\,\,\,k,q\in\{1,...,n\}.
\end{equation}
Using the expression of $A_{kq}$ we deduce that
$u_1+u_2=u_2+u_3=...=u_{n-1}+u_n$ and consequently $u_k=u_q$ if
$k-q\in \,2\mathbb{Z}$. The matrices $B_{kq}$ have the form
$$B_{kq}=\left(%
\begin{array}{cc}
  1 & u_1 \\
  u_1+u_2 & X_{q-1}+X_q+u_1^2 \\
\end{array}%
\right)\,\,\,\texttt{if}\,\,\,q\in 2\mathbb{N}+1$$and
$$B_{kq}=\left(%
\begin{array}{cc}
  1 & u_2 \\
  u_1+u_2 & X_{q-1}+X_q+u_2^2 \\
\end{array}%
\right)\,\,\,\texttt{if}\,\,\,q\in 2\mathbb{N}.$$ We have $det
B_{kq}=0$ if and only if
$X_1=X_1+X_2=...=X_{n-2}+X_{n-1}=X_{n-1}=u_1u_2$.\vspace{2mm}

\noindent {\bf The case $n\in 2\mathbb{N}+1$}. It is easy to
see that $X_1=X_2=...=X_{n-1}=0$ and $u_1u_2=0$. We observe that
$$M_{(1)}^{\mathbb{F}_{23}}=\{(\underbrace{0,...,0}_{n-1},\underbrace{u_1,0,...,u_1,0,u_1}_n)\,|\,u_1\in \mathbb{R}\}\cup\{(\underbrace{0,...,0}_{n-1},
\underbrace{0,u_2,...,0,u_2,0}_n)\,|\,u_2\in\mathbb{R}\}.$$

\noindent {\bf The case $n\in 2\mathbb{N}$}. We deduce that
$X_1=X_3=...=X_{n-1}=X$, $X_2=X_4=...=X_{n-2}=0$ and $u_1u_2=X$.
In this case we have
$$M_{(1)}^{\mathbb{F}_{23}}=\{(\underbrace{X,0,X,...0,X}_{n-1},\underbrace{u_1,u_2,...,u_1,u_2}_{n})\,|\,X=u_1u_2\}.$$

\noindent {\bf The study of $M_{(2)}^{\mathbb{F}_{123}}$}.

\noindent For a point of the set $M_{(2)}^{\mathbb{F}_{123}}$ we have
\begin{equation}\label{rang2}
det(A_{kqr})=0,\,\,\,\det(B_{kqr})=0,\,\,\,det(C_{kqr})=0,\,\,\,\forall
k,q,r
\end{equation}
where
\begin{equation}
A_{kqr}=\left(%
\begin{array}{ccc}
0 & 0 & 1 \\
1 & 1 & u_r \\
u_k+u_{k+1} & u_q+u_{q+1} & X_{r-1}+X_r+u_r^2 \\
\end{array}%
\right),\,\,\,k,q\in\{1,...,n-1\},\,\,\,r\in\{1,...,n\}
\end{equation}

\begin{equation}
B_{kqr}=\left(%
\begin{array}{ccc}
0 & 1 & 1 \\
  1 & u_q & u_r \\
  u_k+u_{k+1} & X_{q-1}+X_q+u_q^2 & X_{r-1}+X_r+u_r^2 \\
\end{array}%
\right),\,\,\,k\in\{1,...,n-1\},\,q,r\in\{1,...,n\}
\end{equation}

\begin{equation}
C_{kqr}=\left(%
\begin{array}{ccc}
1 & 1& 1 \\
u_k & u_q & u_r\\
X_{k-1}+X_k+u_k^2 & X_{q-1}+X_q+u_q^2 & X_{r-1}+X_r+u_r^2 \\
\end{array}%
\right),\,\,\,k,q,r\in\{1,...,n\}
\end{equation}
Using the expression of $A_{kqr}$ we deduce that
$u_1+u_2=u_2+u_3=...=u_{n-1}+u_n$ and consequently $u_k=u_q$ if
$k-q\in \,2\mathbb{Z}$. For $k\in \{1,...,n-1\}$, we have
$$det B_{k,k+1,k}=det \left(%
\begin{array}{ccc}
0 & 0 & 1 \\
1 & u_{k+1}-u_k & u_k \\
u_k+u_{k+1} & X_{k+1}-X_{k-1}+u_{k+1}^2-u_k^2 & X_{k-1}+X_k+u_k^2 \\
\end{array}%
\right)=0$$ and we deduce that $X_{k+1}=X_{k-1}$.\vspace{2mm}

\noindent {\bf The case $n\in 2\mathbb{N}+1$}. It is easy to
see that $X_1=X_2=...=X_{n-1}=0$, $u_1=u_3=...=u_n$ and
$u_2=u_4=...=u_{n-1}$. All the conditions \eqref{rang2} are
verified and the set is
$$M_{(2)}^{\mathbb{F}_{123}}=\{(\underbrace{0,...,0}_{n-1},\underbrace{u_1,u_2,...,u_1}_{n})\,|\,u_1,u_2\in \mathbb{R}\}.$$

\noindent {\bf The case $n\in 2\mathbb{N}$}. It is easy to
see that $X_1=X_3=...=X_{n-1}=X$, $X_2=X_4=...=X_{n-2}=0$,
$u_1=u_3=...=u_{n-1}$ and $u_2=u_4=...=u_n$. All the conditions
\eqref{rang2} are verified and the set is
$$M_{(2)}^{\mathbb{F}_{123}}=\{(\underbrace{X,0,...,X,0,X}_{n-1},\underbrace{u_1,u_2,...,u_1,u_2}_{n})\,|\,X,u_1,u_2\in \mathbb{R}\}.$$

{\bf Acknowledgments}. Petre Birtea has been supported by CNCSIS {
UEFISCSU, project number PN II - IDEI 1081/2008.

\end{document}